\documentclass[10pt]{amsart}
\usepackage[usenames,dvipsnames]{color}
\usepackage{parskip}
\usepackage{amsfonts}
\usepackage{amscd}
\usepackage[centertags]{amsmath}
\usepackage{amssymb}
\usepackage[all,cmtip]{xy}
\usepackage[english]{babel}
\usepackage{tabularx}
\usepackage{mathtools}
\usepackage{tikz-cd}

\usepackage{euscript}
\usepackage[T1]{fontenc}
\usepackage{graphicx}
\usepackage{float}

\newtheorem{theorem}{Theorem}[section]
\newtheorem{corollary}[theorem]{Corollary}

\newtheorem{proposition}[theorem]{Proposition}
\newtheorem{remark}[theorem]{Remark}

\theoremstyle{definition}
\newtheorem{definition}[theorem]{Definition}

\numberwithin{figure}{section}
\numberwithin{table}{section}

\newcommand{\simpd}{\mbox{SimpD}}
\newcommand{\sd}{\mbox{Sd}}
\newcommand{\SC}{\mbox{SC}}
\newcommand{\pr}{\mbox{pr}}
\newcommand{\tc}{\mbox{TC}}
\newcommand{\id}{\mbox{id}}

\newcommand{\D}{\mbox{D}}
\newcommand{\cat}{\mbox{cat}}

\author{Ayse Borat}
\date{\today}

\address{\textsc{Ayse Borat}
Bursa Technical University\\
Faculty of Engineering and Natural Sciences\\
Department of Mathematics\\
Bursa, Turkey}
\email{ayse.borat@btu.edu.tr}

\begin{document}

\title{Simplicial Distance}

\subjclass[2010]{}


\begin{abstract}
In this paper we will introduce and give topological properties of a new concept named simplicial distance which is the simplicial analog of the homotopic distance (in the sense of Marcias-Virgos and Mosquera-Lois in their paper \cite{VM}). According to our definition of simplicial distance, simplicial complexity is a particular case of this new concept. 
\end{abstract}

\maketitle

\section{Introduction}

Simplicial complexity as defined by Gonzalez in \cite{G} and its higher analogs as defined by Borat in \cite{B} have benefits such as the motion planners in this setting can be computed with a help of a computer. In this paper we will introduce a new concept named simplicial distance of which simplicial complexity is a particular case. So one open question in this paper is to carry the work to the digital world.

To understand the simplicial complexity more clearly we need the definition of the topological complexity. Therefore we will start this section by recalling some basic definitions such as $\tc$, $\cat$ and homotopic distance. 

\begin{definition}\cite{VM} Let $f,g: X\rightarrow Y$ be continuous maps. Homotopic distance $\D(f,g)$ between $f$ and $g$ is the least non-negative integer $k$ if there exist open subsets $U_0, U_1,\ldots, U_k$ covering $X$ such that $f|_{U_i}\simeq g|_{U_i}$ for each $i=0,1,\ldots,k$.
\end{definition}

\begin{definition}\cite{FMMV2} Lusternik Schnirelmann category $\cat(X)$ of a space $X$ is the least non-negative integer $k$ if there exist open subsets $U_0, U_1,\ldots, U_k$ covering $X$ such that each inclusion $\iota: U_i \hookrightarrow X$ is nullhomotopic. 
\end{definition}

\begin{definition}\cite{F} Let $\pi:PX\rightarrow X\times X$, $\pi(\gamma)=(\gamma(0),\gamma(1))$ be the path fibration where $PX$ stands for the path space of $X$. Topological complexity $\tc(X)$ is the least non-negative integer $k$ if there exist open subsets $U_0, U_1,\ldots, U_k$ covering $X\times X$ such that $\pi$ admits a continuous section on each $U_i$.
\end{definition}

\begin{definition}\cite{G} Let $K$ be an ordered complex. For $b, c \in \mathbb{Z}_{\geq 0}$, the $(b,c)$-simplicial distance $\SC^b_c(K)$ is the least non-negative integer $k$ if there are subcomplexes $J_0, J_1, \cdots, J_k$ covering $\operatorname{Sd}^b(K)$ such that

\[
\pi_1:J_i\hookrightarrow \operatorname{Sd}^b(K\times K) \xrightarrow{\iota} K\times K \xrightarrow{\pr_1} K
\]
\noindent and 
\[
\pi_2:J_i\hookrightarrow \operatorname{Sd}^b(K\times K) \xrightarrow{\iota} K\times K \xrightarrow{\pr_2} K 
\]

\noindent are $c$-contiguous for each $i\in \{0,1,\cdots, k\}$. \\

If there is no such covering, then we set $\SC^b_c(K)=\infty$.
\end{definition}

\begin{definition}\cite{G} There is a monotonic sequence 
\[
\SC^b_0(K)\geq \SC^b_1(K) \geq \SC^b_2(K) \geq \ldots
\]

\noindent $\SC^b(K)$ is defined to be the stabilized value of this sequence. Another monotonic sequence can be given by 
\[
\SC^0(K)\geq \SC^1(K) \geq \SC^2(K) \geq \ldots
\]

\noindent The \textit{simplicial complexity} $\SC(K)$ is defined to be the stabilized value of the latter sequence. 
\end{definition}

Notice that topological complexity in \cite{F} and Lusternik Schnirelmann category in \cite{FMMV2} are defined in "nonreduced" terms. Throughout this paper we will use the reduced version of homotopic distance, $\tc$, $\cat$ and their simplicial analogs.

If we replace $f=\pr_1:X\times X\rightarrow X$ and $g=\pr_2: X\times X\rightarrow X$ with the projection maps to the factors, then $\D(\pr_1,\pr_2)=\tc(X)$, \cite{VM}. Moreover, one can see homotopic distance as a generalization of Lusternik Schnirelmann category as well, as $\D(\id_X,c)=\cat(X)$ where $c$ is a constant map on $X$, \cite{VM}. One another characterization of LS category via homotopic distance can be given as follows.

For a base point $x_0\in X$, consider the inclusions $i_1,i_2: X\rightarrow X\times X$ given by $i_1(x)=(x,x_0)$ and by $i_2(x)=(x_0,x)$. Then $\D(i_1,i_2)=\cat(X)$, \cite{VM}.

One of the main results of this paper is that we can have a similar relation between $\SC$ and $\simpd$ as we have a relation between $\tc$ and $\D$ in the topological realm, see Proposition~\ref{c4}. We also have a result which gives the relation between $\simpd$ and the $\tc$ of the geometric realization of some particular complex, see Proposition~\ref{c4}. One another main result is that if $\alpha: L \rightarrow M$ is a simplicial map with a left strong homotopy type inverse, then  $\simpd(\alpha\circ \varphi, \alpha \circ \psi)=\simpd(\varphi,\psi)$, see Corollary~\ref{c5}. We have more results in the last section such as how being in the same contiguity class effects the simplicial distance. 

We would like to mention that for more work on homotopic distance one can see \cite{BV} where homotopic distance is introduced in the higher dimensional case and \cite{MVML2} in which homotopic distance between functors is carried out.

\section{Simplicial Distance}

\begin{definition}
Let $\varphi, \psi \colon K \to L$ be two simplicial maps between simplicial complexes. $\varphi$ and $\psi$ are said to be contiguous, if $\sigma=\{v_0,\dotsc,v_n\}$ is a simplex in $K$, then $\varphi(\sigma) \cup \psi(\sigma)=\{ \varphi(v_0), \dotsc, \varphi(v_n), \psi(v_0), \dotsc, \psi(v_n) \}$ constitute a simplex in $L$ .
\end{definition} 

\begin{definition} \cite{G} Let $K, L$ be complexes. $\varphi, \psi: K\rightarrow L$ are called $c$-contiguous, denoted by $\varphi \sim_c \psi$, if there exists a sequence of maps $H_i:K\rightarrow L$, $i=0,1,\ldots, c$ satisfying $H_0=f$, $H_c=g$ and that for each $i=1,2,\cdots, c$, the pair $(H_{i-1},H_i)$ is contiguous. 
\end{definition}

Topological realization of the product of abstract complexes is homeomorphic to the product of the topological realizations of these complexes. So throughout this paper, all complexes and their products will be taken in the category of ordered complexes unless stated otherwise.

For $b\in \mathbb{Z}_{\geq 0}$, we will fix an approximation $\iota: \sd^{b+1}(K)\rightarrow \sd^b(K)$ of the identity map $\id_{||K||}$. Since approximations behave well with respect to compositions (see, for example, Remark 2.4 in \cite{G}), by abuse of notation, we will denote the iterated compositions of the maps between barycentric subdivisions by $\iota: \sd^{b}(K)\rightarrow \sd^{b'}(K)$. 

\begin{definition} Let $K, L$ be an ordered complex and $\varphi, \psi: K\rightarrow L$ be simplicial maps. For $b, c \in \mathbb{Z}_{\geq 0}$, the $(b,c)$-simplicial distance $\simpd^b_c(\varphi,\psi)$ is the least non-negative integer $k$ if there are $k+1$ subcomplexes $J_0, J_1, \cdots, J_k$ covering $\operatorname{Sd}^b(K)$ such that

\[
\gamma_\varphi:J_i\hookrightarrow \operatorname{Sd}^b(K) \xrightarrow{\pi_\varphi} L 
\]
\noindent and 
\[
\gamma_\psi:J_i\hookrightarrow \operatorname{Sd}^b(K) \xrightarrow{\pi_\psi} L 
\]

\noindent are $c$-contiguous for each $i\in \{0,1,\cdots, k\}$ where $\pi_\varphi:\operatorname{Sd}^b(K)\xhookrightarrow{\iota} K\xrightarrow{\varphi} L $ and $\pi_\psi:\operatorname{Sd}^b(K)\xhookrightarrow{\iota} K\xrightarrow{\psi} L $. \\

If there is no such covering, then we set $\simpd^b_c(\varphi,\psi)=\infty$.

\end{definition}

\begin{proposition}\label{P} If $\pr_i:K\times K \rightarrow K$ is the projection to the $i$-th factor, then $\simpd^b_c(\pr_1, \pr_2)= \SC^b_c(K)$.
\end{proposition}
\qed

The following proposition follows by Proposition~\ref{P} and (3.4) in \cite{G}.

\begin{proposition} If $\pr_i:K\times K \rightarrow K$ is the projection to the $i$-th factor, then $\tc(||K||)\leq \simpd^b_c(\pr_1, \pr_2)$.
\end{proposition}
\qed

\begin{remark} For simplicial maps $\varphi, \psi \colon K \to L$, we have the sequence

\[
\simpd^{b}_0(\varphi,\psi)\geq \simpd^{b}_1(\varphi,\psi)\geq \simpd^{b}_2(\varphi,\psi) \geq \ldots 
\]

So we can define 

\[
\simpd^{b}(\varphi,\psi):= \underset{c\to \infty}{\lim} \simpd^b_c(\varphi,\psi).
\]
\end{remark}

\begin{remark}\label{r1}\cite{S} Composites of 1-contiguous maps are 1-contiguous, that is, if  $\varphi,\widetilde{\varphi}: K\rightarrow L$ and $\psi, \widetilde{\psi}:L\rightarrow M$ are 1-contiguous, so is $\psi\circ\varphi\sim_1  \widetilde{\psi}\circ \widetilde{\varphi}$. It follows that we have a well-defined composition of contiguity classes $[\psi]\circ [\varphi]=[\psi\circ\varphi]$.
\end{remark}

\begin{remark}\label{r2}\cite{S} Two approximations to the same continuous maps are 1-contiguous. 
\end{remark}

\begin{remark}\label{r3}\cite[Remark 2.6]{G}\label{r3} If $\varphi$ and $\psi$ are $c$-contiguous, then the compositions

\[
\beta\varphi\alpha, \beta\psi\alpha: J\xrightarrow{\alpha} K\xrightarrow{\varphi,\psi} L\xrightarrow{\beta} M
\]

are $c$-contiguous.
\end{remark}

\begin{proposition} Let $\varphi, \psi: K\rightarrow L$ be simplicial maps. $\simpd^b(\varphi,\psi)$ is independent of the chosen approximations $\iota: \sd^{b}(K)\rightarrow \sd^{b-1}(K)$ of the identity.
\end{proposition}

\begin{proof} Suppose there are two approximations of the identity, namely, $\iota, \tilde{\iota}: \sd^{b}(K)\rightarrow \sd^{b-1}(K)$. Suppose also $\simpd^b(\varphi,\psi)$ and $\widetilde{\simpd}^b(\varphi,\psi)$ denote the invariants obtained from $\iota$ and $\tilde{\iota}$, respectively. 

Consider an inclusion $j:J\rightarrow \sd^{b}(K)$ of some subcomplex $J\subseteq K$. Since $\gamma_\varphi=\pi_\varphi \circ j$ and $\gamma_\psi=\pi_\psi \circ j$ are $c$-contiguous, there are $h_0, h_1, \ldots, h_c: J\rightarrow L$ such that $\pi_\varphi \circ j = h_0$ and $\pi_\psi \circ j =h_c$ and $(h_{i-1}, h_i)$ is 1-contiguous. On the other hand, by Remark~\ref{r2} and Remark~\ref{r3}, $(\pi_{\varphi}\circ j, \widetilde\pi_{\varphi}\circ j)$ and  $(\pi_{\psi}\circ j, \widetilde\pi_{\psi}\circ j)$ are $c$-contiguous pairs. So we may obtain a contiguity chain $\widetilde\pi_{\varphi}\circ j, h_0, h_1, \ldots, h_c, \widetilde\pi_{\psi}\circ j: J\rightarrow L$ of length $c+2$. So $\widetilde{\simpd}_{c+2}^b(\varphi,\psi) \leq \simpd_c^b(\varphi,\psi)$. Proceeding similarly, one can show that $\simpd_{c+2}^b(\varphi,\psi) \leq \widetilde{\simpd}_c^b(\varphi,\psi)$. Hence we have $\simpd^b(\varphi,\psi)=\widetilde{\simpd}^b(\varphi,\psi)$.
\end{proof}

\begin{proposition}\label{p3} There is a sequence 
\[
\simpd^0(\varphi,\psi)\geq \simpd^1(\varphi,\psi)\geq \simpd^2(\varphi,\psi)\geq \ldots
\]
\end{proposition}

\begin{proof} Consider a subcomplex $J\subseteq K$. Let $\lambda: \sd^1(J)\rightarrow J$ be an approximation of the identity on $||J||$. Then we have the following diagram in which the compositions are 1-contiguous since $\alpha$ and $\beta$ are approximations of the inclusions $||J||\rightarrow ||K||$. 

\begin{center}
\begin{tikzcd} 
\operatorname{Sd}^1(J) \arrow[r, "\beta"] \arrow[d, "\lambda"]
& \operatorname{Sd}^{b+1}(K) \arrow[d, "\iota"] \\
J \arrow[r, "\alpha"]
& \operatorname{Sd}^{b}(K)
\end{tikzcd}
\end{center}

Let us consider $\simpd^b_c(\varphi,\psi)$ and suppose that $\gamma_\varphi$ and $\gamma_\psi$ are $c$-contiguous. From the above diagram, we can construct $c+2$-contiguous simplicial maps $\bar{\gamma}_\varphi$ and $\bar{\gamma}_\psi$ such that they are 1-contiguous to $\gamma_\varphi$ and $\gamma_\psi$, respectively. Hence we have $\simpd^b_c(\varphi,\psi)\geq \simpd^{b+1}_{c+2}(\varphi,\psi)$. Therefore

\[
\simpd^b(\varphi,\psi)= \simpd^b_c(\varphi,\psi) \hspace{0.03in}, \hspace{0.3in} \text{for large } c
\]
\[
\geq \simpd^{b+1}_{c+2}(\varphi,\psi)
\]
\[
\hspace{0.04in} \geq \simpd^{b+1}(\varphi,\psi).
\]
\end{proof}

\begin{definition} The stabilized value of the sequence in Proposition~\ref{p3} will be called simplicial distance between $\varphi$ and $\psi$, and be denoted by $\simpd(\varphi,\psi)$. 

\end{definition}

\section{Properties of Simplicial Distance}

The first proposition shows the relation between the simplicial distance and the topological complexity and the relation between the simplicial distance and the simplicial complexity.   

\begin{proposition}\label{c4} For a finite complex $K$, if $\pr_i:K\times K \rightarrow K$ is the projection to the $i$-th factor for $i=1,2$, then we have the following

\begin{itemize}
\item[(a)] $\tc(||K||)= \simpd(\pr_1, \pr_2)$. 
\item[(b)] $\simpd(\pr_1, \pr_2)=\SC(K)$.
\end{itemize}
\end{proposition}

\begin{proof}
(a) follows from Theorem 3.5 in \cite{G} and Proposition~\ref{P} whereas (b) follows from Theorem 3.5 in \cite{G} and (a).
\end{proof}

\begin{corollary} If $K$ and $L$ are two finite complexes whose geometric realizations are homotopy equivalent. Let $\pr_i:K\times K \rightarrow K$ and $\widetilde{pr_i}:L\times L \rightarrow L$ be the projections to the $i$-th factors for $i=1,2$, then $\simpd(\pr_1, \pr_2)=\simpd(\widetilde{pr_1}, \widetilde{pr_2})$.
\end{corollary}

\begin{corollary} Let $\pr_i:K\times K \rightarrow K$ be projections as in Proposition~\ref{c4}, then $||K||$ is contractible iff $\simpd(\pr_1, \pr_2)=0$.
\end{corollary}

\begin{proposition} $\simpd(\varphi,\psi)=\simpd(\psi,\varphi)$.
\end{proposition}
\qed

\begin{theorem}\label{p3} If $\varphi\sim \psi$, then $\simpd(\varphi,\psi)=0$.
\end{theorem}

\begin{proof} If $\varphi\sim \psi$, then there is some non-negative integer $c$ such that $\varphi\sim_c \psi$, that is, there exists $H:K\rightarrow L$ such that $H_0=\varphi$, $H_c=\psi$ and $(H_{i-1},H_i)$ is 1-contiguous for each $i$. If we could show that $\simpd^b_c(\varphi,\psi)=0$, then it immediately follows that $\simpd(\varphi,\psi)=0$. 

Consider the following simplicial maps 

\[
\gamma_\varphi:K\xrightarrow{j} \operatorname{Sd}^b(K)\xhookrightarrow{\iota} K\xrightarrow{\varphi} L 
\]
\noindent and 
\[
\gamma_\psi:K\xrightarrow{j} \operatorname{Sd}^b(K)\xhookrightarrow{\iota} K\xrightarrow{\psi} L 
\]

where $\varphi$ and $\psi$ are $c$-contiguous. Define 

\[
G: K \xrightarrow{j} \operatorname{Sd}^b(K)\xhookrightarrow{\iota} K\xrightarrow{H} L 
\]

We have $G_0=H_0\circ \iota \circ j=\varphi\circ \iota \circ j$ and $G_c=H_c\circ \iota \circ j=\psi\circ \iota \circ j$, and $(G_{i-1}, G_i)$ is 1-contiguous for each $i$ by Remark~\ref{r1}. So $\gamma_\varphi \sim_c \gamma_\psi$. Hence $\simpd(\varphi,\psi)=0$. 

\end{proof}

It is known that if $f,g:||K||\rightarrow ||L||$ are homotopic then their simplicial approximations $\varphi_f,\varphi_g: \sd^N(K)\rightarrow L$ are in the same contiguity class for some integer $N$ (see for instance \cite{S}). By Proposition~\ref{p3}, $\simpd(\varphi_f,\varphi_g)=0$. Notice that this result  is concordant with the result (2) in Subsection 2.1 in \cite{VM}.

\begin{theorem} If $\varphi\sim \varphi'$ and $\psi\sim\psi'$, then $\simpd(\varphi,\psi)=\simpd(\varphi',\psi')$.
\end{theorem}

\begin{proof} First we will show that $\simpd(\varphi,\psi)\leq\simpd(\varphi',\psi')$. Since $\varphi\sim \varphi'$ and $\psi\sim\psi'$, we have $\varphi\sim_c \varphi'$ and $\psi\sim_d\psi'$ for some $c,d \in\mathbb{Z}_{\geq 0}$. 

Suppose $\simpd(\varphi',\psi')=k$. So 
\[
k=\simpd^b_a(\varphi',\psi') \hspace{0.1in} \text{for some large } a,b\in\mathbb{Z}_{\geq 0}. 
\]

By definition of $(b,a)$-simplicial distance, there are $k+1$ subcomplexes $J_i$'s such that for each $i$

\[
\gamma_\varphi':J_i\xrightarrow{j} \operatorname{Sd}^b(K)\xhookrightarrow{\iota} K\xrightarrow{\varphi'} L 
\]
\noindent and 
\[
\gamma_\psi':J_i\xrightarrow{j} \operatorname{Sd}^b(K)\xhookrightarrow{\iota} K\xrightarrow{\psi'} L 
\]

are $a$-contiguous. Hence on each $J_i$, 

\[
\varphi\circ \iota\circ j \sim_c \varphi' \circ \iota \circ j = \gamma_{\varphi'} \sim_a \gamma_{\psi'} = \psi' \circ \iota \circ j \sim_d \psi\circ \iota \circ j
\]

It follows that $\gamma_\varphi \sim_{c+a+d} \gamma_\psi$ where $\gamma_\varphi = \varphi\circ \iota\circ j $ and $\gamma_\psi = \psi\circ \iota\circ j $ . In other words, there exists a simplicial map $G_j:J_i\rightarrow L$ such that $G_0=\gamma_\varphi$, $G_{c+a+d}=\gamma_\psi$ and $(G_{j-1}, G_j)$ is 1-contiguous for each $j$. More precisely, $G_j$ is given as follows.

\[
G_i:=\begin{cases} 
      \bar{H}_j\circ \iota\circ j &, \hspace{0.1in}  0\leq j\leq c \\
     \bar{F}_{j-c} &, \hspace{0.1in}  c\leq j\leq c+a \\
     \bar{B}_{j-c-a}\circ \iota\circ j &, \hspace{0.1in}  c+a\leq j\leq c+a+d 
   \end{cases}
\]
\vspace{0.1in} 

\noindent where $\bar{H}_j$ is the family of maps of length $c$ (between $\varphi$ and $\varphi'$), $\bar{F}_j$ is the family of maps of length $a$ (between $\gamma_{\varphi'}$ and $\gamma_{\psi'}$) and $\bar{B}_j$ is the family of maps of length $d$ (between $\psi$ and $\psi'$).

Therefore $\simpd^b_{c+a+d}(\varphi,\psi)\leq k$ which implies that $\simpd(\varphi,\psi)\leq k$. The other way around can be proved similarly and the required equality holds. 
\end{proof}

The following theorem shows how simplicial distance behave under the composition. 

\begin{theorem}\label{thm2} If $\varphi, \psi: K\rightarrow L$ and $\alpha, \bar{\alpha}: L \rightarrow M$ are simplicial maps and $\alpha \sim_a \bar{\alpha}$ for some $a\in \mathbb{Z}_{\geq 0}$, then 
\[
\simpd(\alpha \circ \varphi, \bar{\alpha} \circ \psi)\leq \simpd(\varphi,\psi).
\]

\end{theorem}

\begin{proof} Let $\simpd(\varphi,\psi)=k$. Then for some large $b,c$
\[
k= \simpd^b_c(\varphi,\psi).
\]

So there exist subcomplexes $J_0, J_1, \ldots, J_k$ covering of $\sd^b(K)$ such that 
\[
\gamma_\varphi:J_i\xrightarrow{j} \operatorname{Sd}^b(K)\xhookrightarrow{\iota} K\xrightarrow{\varphi} L 
\]
\noindent and 
\[
\gamma_\psi:J_i\xrightarrow{j} \operatorname{Sd}^b(K)\xhookrightarrow{\iota} K\xrightarrow{\psi} L 
\]

are $c$-contiguous for each $i$. 

Define 
\[
\alpha \circ \gamma_\varphi:J_i\xrightarrow{j} \operatorname{Sd}^b(K)\xhookrightarrow{\iota} K\xrightarrow{\varphi} L \xrightarrow{\alpha} M
\]
\noindent and 
\[
\bar{\alpha} \circ \gamma_\psi:J_i\xrightarrow{j} \operatorname{Sd}^b(K)\xhookrightarrow{\iota} K\xrightarrow{\psi} L \xrightarrow{\bar{\alpha}} M
\]

In the next argument we will show that these maps are $\operatorname{max}\{a,c\}$-contiguous. 

Since $\gamma_\varphi \sim_c \gamma_\psi$, for each $i$ there exists $H_{j}: J_i \rightarrow L$ such that $H_0=\gamma_\varphi$, $H_c= \gamma_\psi$ and $(H_{s-1}, H_s)$ is 1-contiguous for $s=1,2,\ldots, c$.

Since $\alpha\sim_a \bar{\alpha}$, there exist $F_j: L\rightarrow M$ such that $F_0=\alpha$, $F_a=\bar{\alpha}$ and $(F_{s-1},F_s)$ is 1-contiguous for $s=1,2,\ldots, a$.

Define $\widetilde{F_j}: L\rightarrow M$ by
\[
\widetilde{F_j}=\begin{cases} 
      F_j &, \hspace{0.1in} \text{ if } 0\leq j \leq a \\
      F_a &, \hspace{0.1in} \text{ if } a\leq j \leq \operatorname{max}\{a,c\}\\
   \end{cases}
\]

Define $\widetilde{H_j}: J_i\rightarrow L$ by
\[
\widetilde{H_j}=\begin{cases} 
      H_j &, \hspace{0.1in} \text{ if } 0\leq j \leq c \\
      H_c &, \hspace{0.1in} \text{ if } c\leq j \leq \operatorname{max}\{a,c\}\\
   \end{cases}
\] 

Define $G_j= \widetilde{F_j} \circ \widetilde{H_j}: J_i\rightarrow L \rightarrow M $ so that $G_0=\alpha\circ \gamma_\varphi$ and $G_{\operatorname{max}\{a,c\}}= \bar{\alpha} \circ \gamma_\psi $. Moreover by Remark~\ref{r1}, each $(G_{j-1},G_j)$ is 1-contiguous for $j=1,2,\ldots, \operatorname{max}\{a,c\}$. This gives the $\operatorname{max}\{a,c\}$-contiguity between $\alpha \circ \gamma_\varphi$ and $\bar{\alpha} \circ \gamma_\psi$. Hence $\simpd^b_{\operatorname{max}\{a,c\}}(\alpha\circ \varphi, \bar{\alpha}\circ \psi)\leq k$. Therefore $\simpd(\alpha\circ \varphi, \bar{\alpha}\circ \psi) \leq k$.

\end{proof}

\begin{corollary}\label{c5} Let $\varphi,\psi: K\rightarrow L$ be simplicial maps and let $\alpha: L \rightarrow M$ be a simplicial map with a left strong homotopy type inverse (i.e., $\beta \circ \alpha \sim \id_L$ for some $\beta: M\rightarrow L$). Then $\simpd(\alpha\circ \varphi, \alpha \circ \psi)=\simpd(\varphi,\psi)$.
\end{corollary}

\begin{proof} By Theorem~\ref{thm2}, we have $\simpd(\alpha\circ \varphi, \alpha \circ \psi)\leq\simpd(\varphi,\psi)$. We will show the other way around. Let $\simpd(\alpha\circ \varphi, \alpha \circ \psi)=k$. Then for some large $b,c \in \mathbb{Z}_{\geq 0}$,
\[
k= \simpd^b_c(\alpha\circ\varphi,\alpha\circ\psi).
\]

So there exists $J_0, J_1, \ldots, J_k$ covering $\sd^b(K)$ such that 
\[
\gamma_{\alpha\varphi}:J_i\xrightarrow{j} \operatorname{Sd}^b(K)\xhookrightarrow{\iota} K\xrightarrow{\varphi} L \xrightarrow{\alpha} M
\]
\noindent and 
\[
\gamma_{\alpha\psi}:J_i\xrightarrow{j} \operatorname{Sd}^b(K)\xhookrightarrow{\iota} K\xrightarrow{\psi} L \xrightarrow{\alpha} M
\]

are $c$-contiguous for each $i$. Thus 

\[
\alpha \circ (\varphi \circ \iota) \circ j \sim_c \alpha \circ (\psi \circ \iota) \circ j
\]

\[
\beta \circ \alpha \circ (\varphi \circ \iota) \circ j \sim_c \beta \circ \alpha \circ (\psi \circ \iota) \circ j.
\]

On the other hand, since $\beta \circ \alpha \sim \id_L$, for some $\bar{c}\in \mathbb{Z}_{\geq 0}$ we have $\beta \circ \alpha \sim_{\bar{c}} \id_L$. Hence 
\[
\id_L \circ (\varphi \circ \iota) \circ j \sim_{\bar{c}}  \beta \circ \alpha \circ (\varphi \circ \iota) \circ j \sim_c \beta \circ \alpha \circ (\psi \circ \iota) \circ j\sim_{\bar{c}}\id_L \circ (\psi \circ \iota) \circ j 
\]

Thus 

\[
\id_L \circ (\varphi \circ \iota) \circ j \sim_{2\bar{c}+c} \id_L \circ (\psi \circ \iota) \circ j .
\]
\[
(\varphi \circ \iota) \circ j \sim (\psi \circ \iota) \circ j
\]

It follows that $\simpd^b_c(\varphi,\psi)\leq k$. Therefore $\simpd(\varphi,\psi)\leq k$.
\end{proof}

\end{document}